\documentclass[12pt]{article}
\def\date{16 Sep 2014}
\usepackage{amsmath, amssymb, amsfonts, amsthm, color, epic}

\newtheorem{proposition}{Proposition}[section]
\newtheorem{lemma}[proposition]{Lemma}
\newtheorem{corollary}[proposition]{Corollary}
\newtheorem{theorem}[proposition]{Theorem}
\newtheorem{claim}[proposition]{Claim}

\newtheorem{question}[proposition]{Question}

\newtheorem{thm}[proposition]{Theorem}

{
\theoremstyle{definition}
\newtheorem{definition}[proposition]{Definition}

}

\newcommand{\ch}{{\rm ch}}

\begin{document}
\font\smallrm=cmr8




\baselineskip=12pt
\phantom{a}\vskip .25in
\centerline{{\bf  On the Minimum Edge-Density of 4-Critical Graphs of Girth Five}}
\vskip.4in
\centerline{{\bf Chun-Hung Liu}
\footnote{\texttt{chliu@math.princeton.edu.}}}
\smallskip
\centerline{Department of Mathematics}
\centerline{Princeton University}
\centerline{Princeton, New Jersey, USA}
\vskip.25in

\centerline{{\bf Luke Postle}
\footnote{\texttt{lpostle@uwaterloo.ca. Partially supported by NSERC under Discovery Grant No. 2014-06162}}} 
\smallskip
\centerline{Department of Combinatorics and Optimization}
\centerline{University of Waterloo}
\centerline{Waterloo, Ontario, Canada}

\vskip 1in \centerline{\bf ABSTRACT}
\bigskip

{
\parshape=1.0truein 5.5truein
\noindent

We prove that if $G$ is a $4$-critical graph of girth at least five then $|E(G)| \ge \frac{5|V(G)|+2}{3}$. As a corollary, graphs of girth at least five embeddable in the Klein bottle or torus are $3$-colorable. These are results of Thomas and Walls, and Thomassen respectively. The proof uses the new potential technique developed by Kostochka and Yancey who proved that $4$-critical graphs satisfy: $|E(G)|\ge \frac{5|V(G)|-2}{3}$.
}

\vfill \baselineskip 11pt \noindent \date.
\vfil\eject
\baselineskip 18pt

\section{Introduction}
Intuitively, a graph that has fewers edges can be properly colored by a smaller number of colors.
Kostochka and Yancey~\cite{KY2} confirmed this intuition recently by proving that every non-$3$-colorable graph has large edge density.
We say that a graph is {\it $4$-critical} if it is not $3$-colorable, but all of its proper subgraphs are.

\begin{theorem}[Kostochka, Yancey \cite{KY2}] \label{Ore4}
If $G$ is a $4$-critical graph on $n$ vertices, then $$|E(G)|\ge \frac{5n-2}{3}$$
\end{theorem}

An immediated corollary of this theorem is that every graph of girth at least five that can be embedded in the plane or the projective plane is $3$-colorable. In turn, that statement when combined with a simple argument for identifying vertices on facial 4-cycles implies Gr\"{o}tzsch's Theorem~\cite{Grotzsch}: every triangle-free planar graph is $3$-colorable. In fact, Borodin et al.~\cite{Borodin} outline a few more applications of Theorem~\ref{Ore4} such as Aksenov's Theorem~\cite{Aksenov} that a planar graph with at most three triangles is $3$-colorable.

Furthermore, the bound in Theorem~\ref{Ore4} is tight since it is attained by infinitely many $4$-critical graphs. 
In fact, in a subsequent paper~\cite{KY3}, Kostochka and Yancey characterized the $4$-critical graphs that attain these bounds: $G$ is $4$-critical and $|E(G)|= \frac{5n-2}{3}$ if and only if $G$ is a ``$4$-Ore" graph (defined later).

It is well-known that graphs of large girth can have large chromatic number. However, Gr\"{o}tzsch's Theorem shows that the chromatic number of graphs of large girth can be reduced if we add a topological condition. A natural question is whether Gr\"{o}tzsch's Theorem extends to surfaces of larger genus. 
Unfortunately, this is not true for triangle-free graphs as Gimble and Thomassen~\cite{Gimble} showed that there exist triangle-free 4-critical graphs embeddable in the projective plane. Still, Thomassen~\cite{Thom} showed that projective planar and toroidal graphs of girth five are $3$-colorable.

\begin{thm}[Thomassen \cite{Thom}] \label{torus thm}
Every graph of girth at least five embeddable in the torus or the projective plane is $3$-colorable.
\end{thm}

On the other hand, Thomas and Walls~\cite{KB} then proved the same result for the Klein bottle in a lengthier paper.
 
\begin{thm}[Thomas and Walls \cite{KB}] \label{KB thm}
Every graph of girth at least five embeddable in the Klein bottle is $3$-colorable.
\end{thm}

One may wonder if the topological requirement in these theorems could be replaced by a sparsity condition as in Theorem~\ref{Ore4}. Our main result does just that, improving Kostochka and Yancey's bound for graphs of girth five and thereby generalizing Theorems \ref{torus thm} and \ref{KB thm}. 

\begin{thm} \label{weak main}
If $G$ is a $4$-critical graph of girth at least five, then $|E(G)|\ge \frac{5|V(G)|+2}{3}$.
\end{thm}

\begin{corollary}
Every graph of girth at least five embeddable in the torus or Klein bottle such that all faces have size at least five is $3$-colorable.
\end{corollary}

In fact we prove a more technical but stronger theorem which considers small ``exceptional" 4-critical graphs. One such necessary class is $4$-Ore graphs which we now define:

\begin{definition}
An \emph{Ore-composition} of graphs $G_1$ and $G_2$ is a graph obtained by the following procedure:
\begin{enumerate}
\item delete an edge $xy$ from $G_1$;
\item split some vertex $z$ of $G_2$ into two vertices $z_1$ and $z_2$ of positive degree;
\item identify $x$ with $z_1$ and identify $y$ with $z_2$.
\end{enumerate}
We say that $G_1$ is the \emph{edge-side} and $G_2$ the \emph{split-side} of the composition. Furthermore, we say that $xy$ is the \emph{replaced edge} of $G_1$ and that $z$ is the \emph{split vertex} of $G_2$.
We say that $G$ is a \emph{$k$-Ore graph} if it can be obtained from copies of $K_k$ and repeated Ore-compositions.
\end{definition}

In this paper, $T(G)$ denotes the maximum number of vertex-disjoint cycles of size at most four in a graph $G$.

\begin{thm}\label{Main}
Let $p(G) = 5|V(G)|-3|E(G)|-T(G)$. If $G$ is a $4$-critical graph, then 

\begin{enumerate}
\item $p(G)= +1$ if $G=K_4$,
\item $p(G)= 0$ if $G=H_7$,
\item $p(G)= -1$ if $G=W_5, T_8, T_{11}$ or $G$ is $4$-Ore with $T(G)=3$,
\item $p(G)\le -2$ otherwise, 
\end{enumerate}
where $H_7$ is the $4$-Ore graph on seven vertices, $W_5$ is the graph obtained from the $5$-cycle by adding one vertex adjacent to all other vertices, and $T_8$ and $T_{11}$ are the graphs depicted in Figure~\ref{GraphT}.
\end{thm}

\begin{figure}
\unitlength=1.2mm
\begin{picture}(150, 30)(10,0)

\put(25,30){$T_8$}
\multiput(40,0)(20,0){2}{\circle*{2}}
\put(50,10){\circle*{2}}
\multiput(40,20)(10,0){3}{\circle*{2}}
\multiput(45,30)(10,0){2}{\circle*{2}}
\drawline[1000](40,0)(60,0)(50,10)(40,0)(40,20)(45,30)(55,30)(40,20)
\drawline[1000](50,10)(50,20)(45,30)(60,20)(60,0)
\drawline[1000](50,20)(55,30)(60,20)

\put(95,30){$T_{11}$}
\multiput(110,0)(20,0){2}{\circle*{2}}
\put(120,10){\circle*{2}}
\multiput(110,20)(10,0){3}{\circle*{2}}
\multiput(115,30)(10,0){2}{\circle*{2}}
\put(120,0){\circle*{2}}
\multiput(125,-3)(0,6){2}{\circle*{2}}
\drawline[1000](110,0)(120,0)(125,3)(125,-3)(130,0)
\drawline[1000](120,0)(125,-3)(125,3)(130,0)(120,10)(110,0)(110,20)(115,30)(125,30)(110,20)
\drawline[1000](120,10)(120,20)(115,30)(130,20)(130,0)
\drawline[1000](120,20)(125,30)(130,20)
\end{picture}

\caption{$T_8$ and $T_{11}$}    \label{GraphT}
\end{figure}

Observe that, as all the graphs in Theorem~\ref{Main}(1)-(3) contain triangles, Theorem \ref{weak main} immediately follows from Theorem~\ref{Main}. 

The paper is organized as follows.
In Section 2, we prove some properties for $4$-Ore graphs.
It is a preparation for the rest of the paper.
In Section 3, we introduce the potential technique that is the main tool for proving our main theorem.
In Section 4, we investigate structures of minimum counterexamples of Theorem \ref{Main}.
In Section 5, we complete the proof of Theorem \ref{Main} by the discharging method.
Finally, we mention some concluding remarks in Section 6.

\section{Triangles in $4$-Ore graphs}
We investigate the triangles and $4$-cycles in $4$-Ore graphs in this section. These propositions and lemmas are a necessary preparation for our proof of Theorem \ref{Main}.

\begin{proposition}\label{VertexAndTriangle}
If $H$ is $4$-Ore and $v\in V(H)$, then there exists a triangle in $H\setminus v$.
\end{proposition}

\begin{proof}
We proceed by induction on $|V(H)|$. If $H=K_4$, then every vertex is disjoint from a triangle as desired. So we may suppose that $H\ne K_4$. As $H$ is $4$-Ore, then $H$ is the Ore-composition of two $4$-Ore graphs $H_1$ and $H_2$. Without loss of generality suppose that $H_1$ is the edge-side and $H_2$ is the split-side of the composition. We now have two cases: either $v\in V(H_1)$ or $v\in V(H_2)\setminus V(H_1)$.

First suppose $v\in V(H_1)$. Let $z$ be the split vertex of $H_2$. By induction, there exists a triangle in $H_2\setminus z$, but then that triangle is also in $H\setminus v$ as desired. So we may suppose that $v\in V(H_2)\setminus V(H_1)$. Let $xy$ be the replaced edge of $H_1$. By induction, there exists a triangle in $H_1\setminus x$, but then as $v\ne x$, that triangle is also in $H\setminus v$ as desired.
\end{proof}

\begin{proposition}\label{2ndTriangle}
If $H\ne K_4$ is $4$-Ore and $T$ is a triangle in $H$, then there exists a triangle in $H\setminus V(T)$.
\end{proposition}

\begin{proof}
As $H\ne K_4$ is $4$-Ore, then $H$ is the Ore-composition of two $4$-Ore graphs $H_1$ and $H_2$. Without loss of generality suppose that $H_1$ is the edge-side and $H_2$ is the split-side of the composition. We now have two cases: Since $x$ and $y$ are non-adjacent in $H$, either $T\subseteq H_1$ or $T\subseteq H_2$.

First suppose $T\subseteq H_1$. Let $z$ be the split vertex of $H_2$. By Proposition~\ref{VertexAndTriangle}, there exists a triangle in $H_2\setminus z$, but then that triangle is also in $H\setminus V(T)$ as desired. So we may suppose that $T\subseteq H_2$. Let $xy$ be the replaced edge of $H_1$. As $x$ and $y$ are not adjacent in $H$, we may suppose without loss of generality that $y\not\in T$. By Proposition~\ref{VertexAndTriangle}, there exists a triangle in $H_1\setminus x$, but then that triangle is also in $H\setminus V(T)$ as desired.
\end{proof}

Here is a useful proposition:

\begin{proposition}\label{TOre}
If $H$ is the Ore-composition of two graphs $H_1$ and $H_2$, then $T(H)\ge T(H_1)+T(H_2)-2$. Furthermore, if at least one of $H_1$ or $H_2$ is isomorphic to $K_4$ or $H_7$, then $T(H)\ge T(H_1)+T(H_2)-1$.
\end{proposition}
\begin{proof}
Suppose without loss of generality that $H_1$ is the edge-side with replaced edge $e=xy$ and $H_2$ is the split-side with split vertex $z$. Now $T(H_2-z)\ge T(H_2)-1$ and $H(H_1-e)\ge T(H_1)-1$. However, every $\le 4$-cycle in $H_2-z$ is disjoint from any $\le 4$-cycle in $H_1-e$. Hence $T(H)\ge T(H_2-z) + T(H_1-e) \ge T(H_1)+T(H_2)-2$ as desired.

Furthermore, if $H_1=K_4$ or $H_7$, then $T(H_1)=T(H_1-e)$ and hence $T(H) \ge T(H_1)+T(H_2)-1$. Similarly if $H_2=K_4$, then $T(H_2-z)=T(H_2)$ and hence $T(H)\ge T(H_1)+T(H_2)-1$. Finally suppose that $H_2=H_7$. We are done unless $T(H_1)=T(H_1-e)+1$. This implies that every maximum set of vertex-disjoint $\le 4$-cycles uses the edge $e$. So $T(H_1\setminus \{x,y\})=T(H_1)-1$. Yet every split $H_2'$ of $H_7$ satisifes $T(H_2')=2$. Thus $T(H)\ge T(H_1\setminus \{x,y\}) + T(H_2') = T(H_1) - 1 + 2 = T(H_1)+T(H_2)-1$.
\end{proof}

\begin{corollary}\label{TOre3}
If $H$ is $4$-Ore, then $T(H)=1$ if and only if $H=K_4$. Similarly, $T(H)=2$ if and only if $H=H_7$.
\end{corollary}
\begin{proof}
Let us prove the first statement. Clearly, $T(K_4)=1$. Let $H$ be a $4$-Ore graph with a minimum number of vertices such that $T(H)=1$ and $H\ne K_4$. As $H\ne K_4$, $H$ is the Ore-composition of two graphs $H_1$ and $H_2$. If neither $H_1$ nor $H_2$ is $K_4$, then by the minimality of $H$, $T(H_1),T(H_2)\ge 2$. By Proposition~\ref{TOre}, $T(H)\ge T(H_1)+T(H_2)-2 \ge 2+2-2 =2$, a contradiction. So without loss of generality, we may assume that $H_1=K_4$. But then by Proposition~\ref{TOre}, $T(H)\ge T(H_1)+T(H_2)-1=T(H_2)$. So $T(H_2)=1$ and the minimality of $H$ implies that $H_2=K_4$. Thus $H=H_7$. But $T(H_7)=2$, a contradiction.

Let us prove the second statement. Clearly $T(H_7)=2$. Let $H$ be a $4$-Ore graph with a minimum number of vertices such that $T(H)=2$ and $H\ne H_7$. As $H\ne K_4$, $H$ is the Ore-composition of two graphs $H_1$ and $H_2$. As $H\ne H_7$, at least one of $H_1,H_2$ is not $K_4$. Suppose without loss of generality that $T(H_1)\ge T(H_2)$. Hence $H_1\ne K_4$. 

Suppose $H_1\ne H_7$. By the minimality of $H$, $T(H_1)\ge 3$. If $H_2=K_4$, then by Proposition~\ref{TOre}, $T(H)\ge T(H_1)+T(H_2)-1 \ge 3+1-1 = 3$, a contradiction. If $H_2\ne K_4$, then by Proposition~\ref{TOre}, $T(H)\ge T(H_1)+T(H_2)-2 \ge 3+2-2 = 3$, a contradiction.

So $H_1=H_7$. Hence $T(H_2)\le 2$. By the minimality of $H$, either $H_2=K_4$ or $H_7$. If $H_2=H_7$ then by Proposition~\ref{TOre}, $T(H)\ge T(H_1)+T(H_2)-1 = 2+2-1=3$, a contradiction. So $H_2=K_4$. If $K_4$ is the split-side and $H_7$ is the edge-side of the composition, then $T(H)\ge T(K_4-z) + T(H_7-e) = 1 + 2 = 3$, a contradiction where $e$ is the replaced edge and $z$ is the split vertex. So $K_4$ is the edge-side and $H_7$ is the split side. But then there exist two disjoint $\le 4$-cycles in the split of $H_7$ which do not use both split vertices. Yet there exists a triangle in $K_4-e$ disjoint from either end of the deleted edge. Hence $H$ has three disjoint $\le 4$-cycles and $T(H)\ge 3$, a contradiction.
\end{proof}

We say that a subgraph of graph $H$ isomorphic to $K_4-e$ is a \emph{diamond} of $H$ if the degree three vertices of the $K_4-e$ are also of degree three in $H$.

\begin{proposition}\label{4OreT3}
If $H$ is $4$-Ore with $T(H)=3$ and $H'$ is obtained by splitting a vertex $v$ of $H$ into two vertices $v_1,v_2$, then either
	\begin{enumerate}
		\item $H'$ has a diamond such that neither $v_1$ nor $v_2$ is a vertex of degree three in the diamond, or, 
		\item $T(H')\ge 3$.
	\end{enumerate}
\end{proposition}

\begin{proof}
Suppose that neither (1) nor (2) holds. As $T(H)=3$, $H\ne K_4,H_7$. As $H$ is $4$-Ore and $H\ne K_4$, $H$ is the Ore-composition of two $4$-Ore graphs, an edge-side $H_1$ and a split-side $H_2$. We choose the composition such that $|V(H_1)|$ is as small as possible.

First suppose that $H_1=K_4$. Note then that $H_1-e$ is a diamond in $H$ where $e$ is the replaced edge. If $v\not\in V(H_1)$, (1) holds, a contradiction. So we may assume that $v\in V(H_1)$. Now in every split of a vertex in a diamond, there still exists a triangle or $4$-cycle. Therefore if $T(H_2-z)\ge 2$ (where $z$ is the split vertex of $H_2$), then $T(H')\ge 3$ and (2) holds, a contradiction. So $T(H_2-z)\le 1$. Hence $T(H_2)\le 2$. By Corollary~\ref{TOre3}, $H_2$ is either $K_4$ or $H_7$. Yet $H_2\ne K_4$ as otherwise $H=H_7$, a contradiction. So $H_2=H_7$ and $T(H_2-z)=1$. But then $z$ is a vertex of degree three in $H_7$ and hence there exists a diamond of $H$ contained in $V(H_2)\setminus\{z\}$ and so too in $V(H)-v$ as $v\in V(H_1)$. Thus (1) holds, a contradiction.

Therefore $H_1\ne K_4$. Similarly $H_2\ne K_4$ as otherwise there exists another Ore-composition where the edge-side is $K_4$, contradicting the minimality of $|V(H_1)|$. By Corollary~\ref{TOre3}, $T(H_1),T(H_2)\ge 2$. If neither $H_1$ nor $H_2$ equals $H_7$, then by Proposition~\ref{TOre}, $T(H)\ge T(H_1)+T(H_2)-2 \ge 3+3-2=4$, a contradiction. If exactly one of $H_1$ or $H_2$ equals $H_7$, then by Proposition~\ref{TOre}, $T(H)\ge T(H_1)+T(H_2)-1 \ge 4$, a contradiction. So we may suppose that $H_1$ and $H_2$ are both isomorphic to $H_7$. But then $H_1-e$ contains a diamond in $H$, where $e$ is the replaced edge of $H_1$. But then $H$ is an Ore-composition where the edge-side equals $K_4$, contradicting the minimality of $|V(H_1)|$. 

\end{proof}

\begin{proposition}\label{SameOrK4e}
If $H=T_{8}, T_{11}$ or $4$-Ore with $T(H)=3$ and $f$ is an edge of $H$, then either $T(H)=T(H-f)$ or there exist $K_4-e\subseteq H-f$.
\end{proposition}

\begin{proof}
If $H=T_{11}$, then $K_4-e \subseteq H-f$ since $T_{11}$ has two disjoint copies of $K_4-e$.
When $H=T_8$, $f$ must be the edge incident with both vertices of degree four as otherwise $H-f$ contains a $K_4-e$ as desired. 
But then $T(H-f)=2=T(H)$ as desired.

So we may assume that $H$ is $4$-Ore with $T(H)=3$. Let $u,v$ be the ends of $f$.
Let $H_v$ be the graph obtained from $H$ by splitting $v$ into two vertices $v_1,v_2$ such that $v_2$ has degree one and $u$ is the neighbor of $v_2$.
Then $T(H-f)=T(H_v)$, and $K_4-e \subseteq H-f$ if and only if $K_4-e \subseteq H_v$.
Hence by Proposition~\ref{4OreT3}, either $K_4-e\subseteq H_v$ or $T(H_v)\ge 3$. The former implies that $K_4-e\subseteq H-f$ as desired and the latter implies that $T(H)=T(H-f)$ as desired.
\end{proof}

\section{Potential}

We follow Kostochka and Yancey's proof of Theorem~\ref{Ore4}. However, we modify their definition of potential by subtracting $T(G)$ as follows:

\begin{definition}
Let $G$ be a graph. The \emph{potential} of $G$, denoted by $p(G)$, is $5|V(G)|-3|E(G)|-T(G)$.\\
Let $R\subseteq V(G)$. The \emph{potential} of $R$, denoted by $p_G(R)$ is $p(G[R])$.
\end{definition}

\begin{definition}
If $R\subsetneq V(G)$ with $|R|\ge 4$, and $\phi$ is a $3$-coloring of $G[R]$, we define the \emph{$\phi$-identification of $R$ in $G$}, denoted by $G_{\phi}(R)$, to be the graph obtained from $G$ by identifying for each $i\in\{1,2,3\}$ the vertices colored $i$ in $R$ to a vertex $x_i$, adding the edges $x_1x_2,x_1x_3,x_2x_3$ and then deleting parallel edges. 
We say that $\{x_1x_2,x_1x_3,x_2x_3\}$ is the {\it triangle corresponds to $R$}.
\end{definition}

\begin{proposition}[Claim 8~\cite{KY2}]\label{Phi}
If $G$ is $4$-critical, $R\subsetneq V(G)$ with $|R|\ge 4$, and $\phi$ is a $3$-coloring of $G[R]$, then $\chi(G_{\phi}(R))\ge 4$.
\end{proposition}

Since the resulting graph contains a $4$-critical graph, we may extend the set $R$ to a larger set as follows:

\begin{definition}
Let $G$ be a $4$-critical graph, $R\subsetneq V(G)$ with $|R|\ge 4$ and $\phi$ a $3$-coloring of $G[R]$. Now let $W$ be a $4$-critical subgraph of $G_{\phi}(R)$ and $T$ the triangle corresponding to $R$ in $G$. Then we say that $R' = (V(W)-V(T))\cup R$ is the \emph{critical extension} of $R$ with \emph{extender} $W$. We call $W\cap T$ the \emph{core} of the extension. If in $G$ a vertex in $W-V(T)$ has more neighbors in $R$ than in $V(W \cap T)$ or there exists an edge in $G[V(W)-V(T)]$ that is not in $W-V(T)$, then we say that the extension is \emph{incomplete}. Otherwise, we say the extension is \emph{complete}. If $R'=V(G)$, we say the extension is \emph{spanning}. If the extension is both complete and spanning, then we say it is \emph{total}.
\end{definition}

Note that every critical extension has a non-empty core as otherwise $G$ would contain a proper non-$3$-colorable subgraph contradicting that $G$ is $4$-critical. The following lemma bounds the potential of critical extensions in terms of the original set and the extending critical graph.

\begin{lemma}\label{Extension}
If $G$ is a $4$-critical graph, $R\subsetneq V(G)$ with $|R|\ge 4$ and $R'$ is a critical extension of $R$ with extender $W$ and core $X$, then
$$p_G(R')\le p_G(R) + p(W) - f(|X|) +T(W)-T(W\setminus X),$$
where $f(|X|)=5/7/6$ when $|X|=1/2/3$ respectively.\\
Furthermore, 
$$p_G(R')\le p_G(R)+p(W)-3.$$
\end{lemma}

\begin{proof}
Each vertex of $G[R']$ is a vertex of $G[R]$ or $W\setminus X$, while each of edge of $G[R]$ and $W-E(G_{\phi}[X])$ is an edge of $G[R']$. 
So $\lvert R' \rvert = \lvert R \rvert + \lvert V(W) \rvert - \lvert X \rvert$, and $\lvert E(G[R']) \rvert \geq \lvert E(G[R]) \rvert + \lvert E(W) \rvert - {{\lvert X \rvert}\choose{2}}$.
Furthermore, $T(R')\ge T(R)+T(W\setminus X)$.
Therefore, 
$$p_G(R') \leq p_G(R)+p(W) - 5\lvert X \rvert + 3 {\lvert X \rvert \choose 2} + T(W) - T(W \setminus X).$$
Note that $f(X) = 5\lvert X \rvert - 3 {\lvert X \rvert \choose 2}$.
Observe that $T(W)-T(W\setminus X) \le |X|$. 
Hence $p_G(R')\le p_G(R)+p(W)-3$.
\end{proof}

\section{Structures of a Minimum Counterexample}
In this section, we prove that every minimum counterexample of Theorem~\ref{Main} has certain structures.
We call the graphs in the first three statements of Theorem~\ref{Main} \emph{exceptional}.

For the rest of the paper, let $G$ denote a counterexample of Theorem~\ref{Main} with the minimum number of vertices.
Recall that Kostochka and Yancy prove that $5\lvert V(H) \rvert - 3\lvert E(H) \rvert=2$ for every $4$-Ore graph $H$.
So it suffices to prove the fourth statement of Theorem~\ref{Main}. 
Hence, $p(G)\ge -1$, and $G$ is not exceptional.

\begin{claim}\label{ExtDec}
If $R\subsetneq V(G)$, $|R|\ge 4$ and $R'$ is a critical extension of $R$, then $p_G(R)\ge p_G(R')+2$. Furthermore $p_G(R')\ge p(G)$ and hence $p_G(R)\ge p(G)+2$.
\end{claim}

\begin{proof}
Suppose that $R'$ is a critical extension with extender $W$. 
As $G$ is a minimum counterexample, $p(W)\le 1$. 
By Lemma~\ref{Extension}, $p_G(R')\le p_G(R) + p(W) - 3 \le p_G(R) -2$ as desired. 
By repeatedly applying this result to further critical extensions, we find that $p(G)\le p_G(R') \le p_G(R)-2$.
\end{proof}

\begin{claim}\label{NoK4e}
$G$ does not contain $K_4-e$ as a subgraph.
\end{claim}

\begin{proof}
Suppose not. Choose $R\subseteq V(G)$, $|R|=4$ with $K_4-e \subseteq G[R]$ such that the number of vertices of degree three in the $K_4-e$ which are also of degree three in $G$ is maximized. Note as $G\ne K_4$, $R\subsetneq V(G)$ and $G[R] = K_4-e$.
 
As $|R|\ge 4$, there exists a critical extension $R'$ of $R$. 
Let $W$ be an extender of the extension with core $X$.
Note that no vertex in $G-R$ is adjacent to both ends of $e$, otherwise, the graph obtained from $G$ by removing an edge between the ends of $e$ and their common neighbors in $G-R$ is still not $3$-colorable.

First, assume that $\lvert X \rvert=1$.
Since $G[R']$ is isomorphic to the Ore-composition of $W$ and $K_4$ with $K_4$ as the edge side, we know that $G[R']$ is $4$-critical, and hence $G=G[R']$. Since $-1\le p(G)=p_G(R')\le p_G(R)+p(W)-f(|X|)+|X|\le 4+p(W)-4=p(W)$ by Lemma~\ref{Extension}, $W$ is exceptional by the minimality of $G$.

If $W$ is $4$-Ore, then $G$ is $4$-Ore, so $-1 \leq p(G) = 2-T(G)$.
In other words, $G$ is $4$-Ore with $T(G) \leq 3$, which is exceptional, a contradiction.
Hence, $W=W_5,T_8$ or $T_{11}$.
In these cases, $-1 \leq p(G) = p(W)+T(W)-T(G) = -1+T(W)-T(G)$, so $T(G) \leq T(W)$.
This implies that $W=T_8$ and $W \cap X$ is a vertex in the triangle in $W$ containing no vertices of degree four, and $G=T_{11}$, a contradiction.
So $\lvert X \rvert \geq 2$.

Next assume that $\lvert X \rvert=2$. By Lemma~\ref{Extension}, $p_G(R')\le p_G(R)+p(W)-5$. As $p_G(R)=4$, we have that $p_G(R')\le p(W)-1$. So $p(W)\ge p_G(R')+1$. Yet by Claim~\ref{ExtDec}, $p_G(R')\ge p(G)$. Thus $p_G(W)\ge p(G) + 1 \ge -1+1 = 0$. Thus $W=K_4$ or $H_7$.  If $W=K_4$, then $G=W_5$, a contradiction. So $W=H_7$ and $p(W)=0$. Note however that $T(W)-T(W\setminus X)\le 1$ when $\lvert X \rvert=2$ since it is impossible to remove all triangles in $H_7$ by deleting two adjacent vertices. Therefore, by Lemma~\ref{Extension}, $p_G(R')\le p_G(R)+p(W) - 7 + 1 = 4 + 0 - 7 +1 =-2$. Hence $p(G)\le -2$, a contradiction.

Finally we may assume that $\lvert X \rvert=3$. We claim that $G[R']=G$ when $\lvert X \rvert=3$. Otherwise $G[R']$ is a proper subgraph of $G$ and hence has a $3$-coloring. But then since the ends of $e$ receive the same color in any $3$-coloring of $G[R']$, this would induce a $3$-coloring of $W$, a contradiction. 

Note that no vertex in $G-R$ is adjacent to the both ends of $e$. Hence $G$ is obtained from $W$ by splitting a vertex $x$ in a triangle $T=xyz$ into two vertices $x_1,x_2$ such that $N(x_1)\cup N(x_2)=N(x)$ and $N(x_1)\cap N(x_2)=\{y,z\}$. As every vertex in $G$ has degree at least three, it follows that the degree of $x$ in $W$ is at least four. 

Observe that $|E(G)|=|E(W)|+2$ and $|V(G)|=|V(W)|+1$. Yet $T(G)\ge T(W)-1$ since at most one $4$-cycle or triangle was destroyed by splitting $x$. Hence $p(G)\le p(W)+5-2\cdot 3 + 1 = p(W)$. As $p(G)\ge -1$, we have that $p(W)\ge -1$ and $W$ is exceptional.

Note that $W\ne K_4$ as all vertices of $K_4$ have degree three. If $W=H_7$, then $x$ is the unique vertex of degree four in $H_7$. But then there is only one split of $x$ up to symmetry and in that case $G$ is isomorphic to $T_{8}$, a contradiction.

So $p(W)=-1$. But then the calculations above imply equality throughout and hence $T(G)=T(W)-1$. This in turn implies that there are two disjoint $\le 4$-cycles in $W$, one using $x$ and the other using one of $y$ or $z$. Hence $W\ne W_5$ since $T(W_5)=1$. Similarly, $W\ne T_8,T_{11}$ since there does not exist a triangle in those graphs intersecting two such disjoint cycles of length at most four. 

So $W$ is $4$-Ore with $T(W)=3$ and $T(G)=T(W)-1=2$. Note that $G'=G-\{x_1y,x_2z\}$ can be obtained from $W$ by splitting the vertex $x$. By Proposition~\ref{4OreT3}, either $T(G')\ge 3$ or there exists a diamond in $G'$ such that neither $x_1$ nor $x_2$ is a vertex of degree three in the diamond. If $T(G')\ge 3$, then $T(G)\ge 3$, a contradiction. So there exists a diamond $H$ in $G'$ such that $x_1,x_2$ are not vertices of degree three in $H$. Furthermore, it follows that $y,z$ are not vertices of degree three in $H$ as otherwise $V(H)=\{x_1,x_2,y,z\}$ which does not induce a diamond in $G'$. But then $H$ is also a diamond in $G$ and hence $V(H)$ contradicts the choice of $R$ since the vertices of degree three in $R$ do not remain degree three in $G$ since $\lvert X \rvert = 3$.
\end{proof}

\begin{claim}\label{Pot4}
If $R\subsetneq V(G)$, $|R|\ge 4$, then $p_G(R)\ge p(G)+4$.
Furthermore, $p_G(R) \ge p(G)+5$ unless $G\setminus R$ is a single vertex of degree three in $G$ or contains a triangle consisting of vertices of degree three.
\end{claim}

\begin{proof}
As $R$ is a proper subset of $V(G)$ with $|R|\ge 4$, $R$ has a critical extension $R'$ with extender $W$ and core $X$. 
By Lemma~\ref{Extension}, $p_G(R')\le p_G(R)+p(W)-f(|X|) +T(W)- T(W\setminus X)$. 
By Claim~\ref{ExtDec}, $p(G)\le p_G(R')$.

First suppose that $W$ is not exceptional. 
By the minimality of $G$, $p(W)\le -2$. 
But then $p(G)\le p_G(R) -2 -f(|X|) + |X|$ which is at most $p_G(R)-5$ as desired. 
So we may assume that $W$ is exceptional.

Suppose that $W=K_4$. 
Then $T(W)-T(W\setminus X)$ is $0$ if $|X|=1$ and $1$ if $|X|=2$ or $3$. 
Hence $p_G(R')\le p_G(R)+1-5/7/6 + 0/1/1 = p_G(R) - 4/5/4$. 
Furthermore if the extension is not spanning, then $p(G)\le p_G(R')-2$ by Claim~\ref{ExtDec} and hence $p(G)\le p_G(R)-6$ as desired. 
Similarly, if the extension is incomplete, then $p(G) \leq p_G(R)-6$.
Thus we are done unless the extension is total and $|X|=1$ or $3$. 
When $\lvert X \rvert=1$, $G\setminus R$ must be a triangle consisting of vertices of degree three, while $\lvert X \rvert=3$ implies that $G\setminus R$ is a vertex of degree three in $G$ as desired since the extension is total.

Suppose that $W=H_7$.
Note then that $T(W)-T(W\setminus X) \leq 1$ for any size of $X$. 
As $p(H_7)=0$ and $T(H_7)=2$, $p(G)\le p_G(R) -5/7/6 + 1 = p_G(R) - 4/6/5$. 
So $|X|=1$ and $T(W)-T(W \setminus X)=1$. 
But then $G$ contains a $K_4-e$, contradicting Claim~\ref{NoK4e}. 

Suppose that $W=W_5$. 
As $p(W_5)=-1$ and $T(W_5)=1$, $p(G)\le p_G(R) -1 - 5/7/6 + 1 \le p_G(R)-5$ as desired.

Suppose that $W=T_{8}$ or $T_{11}$. 
As $p(W)=-1$ and $T(W)=2$, $p(G)\le p_G(R) -1 -5/7/6 + 1/2/2 \le p_G(R)-5$ as desired.

Finally we may suppose that $W$ is $4$-Ore and $T(W)=3$. 
As $p(W)=-1$ and $T(W)=3$, $p(G)\le p_G(R) -1 -5/7/6 + 1/2/3 = p_G(R) - 5/6/4$. 
So we are undone unless $|X|=3$. 
However, in that case it follows from Proposition~\ref{2ndTriangle} that $T(W-X)\ge 1$. 
Hence $T(W)-T(W-X)\le 2$. 
Consequently, $p(G)\le p_G(R) -1 -6 + 2 = p_G(R)-5$ as desired.
\end{proof}

\begin{definition}
We say $u,v\in V(G)$ is an \emph{identifiable pair} in a proper subset $R$ of $V(G)$ if $G[R]+uv$ is not $3$-colorable.
\end{definition}


\begin{claim}\label{NoIdPair}
There does not exist an identifiable pair in a proper subset of $V(G)$.
\end{claim}

\begin{proof}
Suppose not. Let $u,v$ be an identifiable pair in a proper subset $R$ of $V(G)$. Since $G[R]+uv$ is not $3$-colorable, there exists
 a $4$-critical subgraph $K$ of $G[R]+uv$. 
As $G$ is $4$-critical, $u,v\in V(K)$ and $uv\in E(K)$.  
Note that $p_G(V(K))= p(K-uv)\le p(K)+4$ since at most one edge is deleted and at most one triangle or $4$-cycle is lost by that edge deletion.
On the other hand, $p_G(V(K)) \geq p(G) +4 \geq 3$ by Claim~\ref{Pot4}.

First suppose that $K$ is not exceptional. 
By the minimality of $G$, $p(K)\le -2$. 
But then $p_G(V(K))\le -2+4 = 2$, a contradiction. 
So $K$ is exceptional.
It follows from Claim~\ref{NoK4e}, that $K\ne K_4, H_7, W_5,T_{11}$ since $K-uv \subseteq G$. 
So $K=T_8$ or $K$ is $4$-Ore with $T(K)=3$.
By Proposition~\ref{SameOrK4e}, either $K_4-e$ is a subgraph of $K-uv$, contradicting Claim~\ref{NoK4e}, or $T(K)=T(K-uv)$. 
The latter implies that $p_G(V(K))\le p(K)+3 = -1 + 3 = 2$, a contradiction. 
%
\end{proof}

\begin{claim}\label{OneInATriangle}
Every triangle in $G$ contains at most one vertex of degree three.
\end{claim}
\begin{proof}
Suppose there exists a triangle $T$ containing two vertices $u,v$ of degree three. Let $a$ be the neighbor of $u$ not in $T$. Let $b$ be the neighbor of $v$ not in $T$. Since $G$ does not contain a $K_4-e$ by Claim~\ref{NoK4e}, $a\ne b$. Yet $G\setminus V(T) + ab$ is not $3$-colorable as otherwise $G$ is $3$-colorable. Thus $a,b$ is an identifiable pair in $V(G-T)$, contradicting Claim~\ref{NoIdPair}.
\end{proof}

Claim~\ref{OneInATriangle} allows us to strengthen the outcome of Claim~\ref{Pot4} as follows:

\begin{claim}\label{Pot5}
If $R\subsetneq V(G)$, $|R|\ge 4$ and $R'$ is a critical extension of $R$, then $p_G(R) \ge p_G(G)+5$ unless $G\setminus R$ is a single vertex of degree three in $G$.
\end{claim}

Similarly we can now exclude all cycles of vertices of degree three.
We define $D_3(G)$ to be the subgraph of $G$ induced by the vertices of degree three.
%
%

\begin{claim}\label{Acyclic}
$D_3(G)$ is acyclic.
\end{claim}
\begin{proof}
Suppose that $D_3$ contains a cycle $C$. Then every two distinct vertices $u,v$ in $N(C)$ is an identifiable pair in $V(G-C)$. Hence by Claim \ref{NoIdPair}, $|N(C)|=1$. But then $G$ contains a $K_4-e$ subgraph contradicting Claim~\ref{NoK4e}. 
\end{proof}


\begin{definition}
The \emph{$H_7$-gadget} is the graph shown in Figure~\ref{GraphH7}. We say $u$ is the \emph{end} of the $H_7$-gadget. Note $H\setminus u$ is obtained from $H_7$ by splitting the vertex of degree four in $H_7$.
\end{definition}

\begin{figure}
\unitlength=1.2mm
\begin{picture}(40, 40)(-50,0)

\multiput(0,0)(40,0){2}{\circle*{2}}
\multiput(10,10)(20,0){2}{\circle*{2}}
\multiput(0,20)(40,0){2}{\circle*{2}}
\multiput(20,20)(0,10){3}{\circle*{2}}
\drawline[1000](0,0)(40,0)(20,20)(0,0)(0,20)(20,40)(40,20)(40,0)
\drawline[1000](0,20)(10,10)
\drawline[1000](20,20)(20,40)
\drawline[1000](30,10)(40,20)
\put(22,20){$a$}
\put(22,30){$u$}
\put(22,40){$b$}

\end{picture}

\caption{$H_7$-gadget with end $u$.}    \label{GraphH7}
\end{figure}
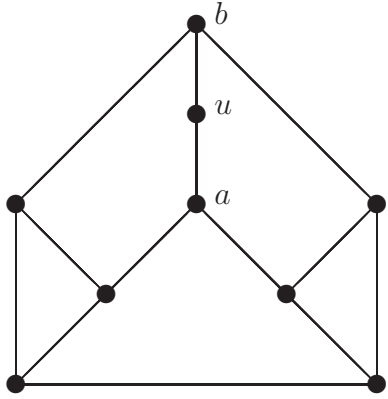

\begin{claim}\label{Gadget}
If $u$ is a vertex of degree three in $G$ with neighbors $a,b,v$, $v$ is of degree three in $G$ and adjacent to another vertex $w\ne u$ of degree three in $G$, then either $a$ is adjacent to $b$, or $a$ and $b$ are in an $H_7$-gadget with end $u$ not containing $v$.
\end{claim}

\begin{proof}
Suppose not. 
That is, $a$ is not adjacent to $b$ and yet they are not in an $H_7$-gadget with end $u$. 
Let $G'$ be obtained from $G$ by deleting $u$ and identifying $a$ and $b$ to a vertex $c$. 
Note that $G'$ is not $3$-colorable and hence contains a $4$-critical subgraph $K$.
Observe that $K$ contains $c$.

Let $R=(V(K)-\{c\}) \cup \{a,b,u\}$. Since $G[R]$ contains two more vertices and at least two more edges than $K$, it follows that $p_G(R)\le p(K)+4 + T(K)-T(G[R]) \le p(K)+5$. 
Moreover, $v,w\not\in K$, since $v$ has degree at most two in $G'$ and $w$ has degree at most two in $G'-v$. 
Hence $\lvert V(G)\setminus R \rvert \geq 2$.
By Corollary~\ref{Pot5}, $p_G(R) \geq p(G)+5 \geq 4$.
So $p(K) \geq -1$ and $K$ is exceptional. Furthermore, if $p(K)=-1$, then $T(K)=T(G[R])+1$.
 
Let $H=G[R]$. Note that $H\setminus\{c\}$ is obtained from $K$ by splitting a vertex. If $K=K_4$, then $H$ contains $K_4-e$ as a subgraph, contradicting Claim~\ref{NoK4e}. If $K=H_7$, then $H$ contains $K_4-e$ as a subgraph, contradicting Claim~\ref{NoK4e}, unless the vertex of degree four is split in such a way that $H$ is an $H_7$-gadget with end $u$, a contradiction.

So $p(K)=-1$. Hence, $T(K)-T(H)=1$.
By Proposition \ref{4OreT3}, $K$ is not $4$-Ore with $T(K)=3$, otherwise $H$ contains a $K_4-e$, contradicting Claim~\ref{NoK4e}, or $T(H)\ge 3=T(K)$, a contradiction.
However, if $K=T_8$ or $T_{11}$, then either $T(K)=T(H)$ or $H$ contains a $K_4-e$ subgraph, which in either case yields a contradiction.
\end{proof}

\begin{claim}\label{Path5}
$G$ does not contain a path of five vertices of degree three.
\end{claim}

\begin{proof}
Let $P=v_1v_2v_3v_4v_5$ be a path of vertices of degree three in $G$. 
Let $x_3$ be the neighbor of $v_3$ not in $P$ (this exists by Claim~\ref{Acyclic}). 
By Claim~\ref{OneInATriangle}, $v_3$ is not in a triangle. 
Hence by Claim~\ref{Gadget} there exists an $H_7$-gadget $H$ with end $v_3$ containing $v_2$ and $x_3$. 
However, as $v_2$ is of degree three, it follows that $v_1$ is in $H$. 
Indeed $v_1$ is in a triangle. 

Yet we also find by Claim~\ref{OneInATriangle}, $v_2$ is not in triangle. 
Let $x_2$ be the neighbor of $v_2$ not in $\{v_1,v_3\}$.
Hence by Claim~\ref{Gadget}, $G$ contains an $H_7$-gadget $H'$ with end $v_2$ containing $v_1$ and $x_2$. 
Since $v_1$ is not in a triangle in $H'$ and $v_1$ is of degree three, this implies that there is an edge $f$ between the two neighbors of $v_1$ distinct from $v_2$ that is not in $H'$. 
But then since $|V(H')|=9$, $|E(H'+f)|=14$ and $T(H')=2$, it follows that $p_G(V(H'))\le 5 \cdot 9-3 \cdot 14 -2 = 1$, a contradiction.
\end{proof}

\begin{claim}
If $C$ is a component of $D_3(G)$, then $|C|\le 6$. 
\end{claim}

\begin{proof}
By Claim~\ref{Acyclic}, $C$ is a tree. 
By Claim~\ref{Path5}, $C$ does not contain a path on five vertices. 
It follows that $C$ has diameter at most three. 
As all vertices in $C$ have degree at most three, we find that $|C|\le 6$. 
\end{proof}

\section{Discharaging}

In this section, we prove Theorem~\ref{Main}.

\begin{definition}
Define the \emph{charge} of a vertex $v$ of $G$, denoted by $ch(v)$, as follows:

$$\ch(v) = 5-\frac{3}{2}d(v),$$
where $d(v)$ denotes the degree of $v$.
\end{definition}
 
Thus if $d(v)=3$, then $\ch(v)=+\frac{1}{2}$ and similarly if $d(v)=4$, then $\ch(v)=-1$. Note that 

$$\sum_v \ch(v) = p(G)+T(G).$$

We will show that $\sum_v \ch(v) \le 0$ using the following discharging rule:

\begin{definition}[Discharging Rule]
Let $v$ be a vertex of degree three in $G$ and $C$ be the component of $D_3(G)$ containing $v$.

\begin{enumerate}
\item If $|C|=1$, then $v$ sends $+\frac{1}{6}$ charge to each neighbor.
\item If $|C|=2$, then $v$ sends $+\frac{1}{4}$ charge to each neighbor of degree at least four.
\item If $C$ has diameter two, then $v$ sends to each neighbor of degree at least four

\begin{enumerate}
\item $+\frac{1}{6}$ charge if $v$ is a non-leaf vertex,
\item $+\frac{1}{3}$ charge if $v$ is a leaf vertex. 
\end{enumerate}

\item If $C$ has diameter three, then $v$ sends to each neighbor of degree at least four
\begin{enumerate}
\item $+\frac{1}{4}$ charge if $v$ is a non-leaf vertex,
\item $+\frac{3}{8}$ charge if $v$ is a leaf vertex. 
\end{enumerate}

\end{enumerate}

Let $\ch^*(v)$ denote the final charge of $v$ after applying the above discharging rules.
\end{definition}

\begin{claim}
If $C$ is a component of $D_3(G)$, then $\sum_{v\in V(C)} \ch^*(v) \le 0$.
\end{claim}
\begin{proof}
First suppose that $|C|=1$. Then $C$ consists of a single vertex $v$. By rule 1, $v$ sends $+\frac{1}{6}$ charge to each of its neighbors. Hence $\ch^*(v)=\frac{1}{2}-3 \cdot \frac{1}{6}=0$ as desired.

Second suppose that $|C|=2$. By rule 2, each vertex in $C$ sends $+\frac{1}{4}$ to each neighbor of degree at least four. Since each vertex in $C$ has two such neighbors, we find that $\sum_{v\in C} \ch^*(v) = 2\cdot\frac{1}{2} - 4\cdot\frac{1}{4} = 0$.

Third suppose that $C$ has diameter two. By rule 3, each non-leaf vertex sends $\frac{1}{6}$ if it has a neighbor of degree at least four, while each leaf vertex sends $+\frac{1}{3}$ to each neighbor of degree at least four (of which it has two). 
Thus if $|C|=3$, $\sum_{v\in C} \ch^*(v) = 3 \cdot \frac{1}{2} -\frac{1}{6} - 4 \cdot \frac{1}{3}=0$. 
If $|C|=4$, then $\sum_{v\in C} \ch^*(v) = 4 \cdot \frac{1}{2} - 6 \cdot \frac{1}{3} = 0$.

Finally we suppose that $C$ has diameter three. 
By rule 4, each non-leaf vertex sends $+\frac{1}{4}$ and each leaf vertex sends $+\frac{3}{8}$ to each neighbor of degree at least four. 
Thus if $|C|=4$, $\sum_{v\in C} \ch^*(v) = 4\cdot\frac{1}{2} - 2\cdot\frac{1}{4} - 4\cdot\frac{3}{8} = 0$. 
If $|C|=5$, then $\sum_{v\in C} \ch^*(v) = 5\cdot\frac{1}{2} - \frac{1}{4} - 6\cdot\frac{3}{8} = 0$. 
If $|C|=6$, then $\sum_{v\in C} \ch^*(v) = 6\cdot\frac{1}{2} - 8\cdot\frac{3}{8} = 0$ as desired.
\end{proof}

\begin{claim}
If $v$ is a vertex of $G$ with $d(v)\ge 4$, then $\ch^*(v)\le 0$.
\end{claim}
\begin{proof}
Suppose that there exists a vertex $v$ with $d(v) \geq 4$ and $\ch^*(v)>0$.
Note that $v$ receives at most $\frac{3}{8}$ from each neighbor of degree three. Thus if $d(v)\ge 5$, $\ch^*(v) \le 5-\frac{3}{2}d(v) +\frac{3}{8}d(v) = 5-\frac{9}{8}d(v) <0$. So $d(v)=4$. 

Since $\ch(v)=-1$ and $v$ receives at most $+\frac{3}{8}$ charge from each neighbor of degree three it follows that $v$ has at least three neighbors of degree three.
In addition, $v$ receives charge from a vertex under at least one of 3(b) or 4(b). 

So first suppose that $v$ receives charge under rule 4(b). 
Then there exists a path $u_1u_2u_3u_4$ in $D_3(G)$ such that $v$ is adjacent to $u_1$.
Applying Claim \ref{Gadget} by taking $u=u_2$, $a=u_1$, and $v=u_3$, we have that $a$ and $u$ are in a triangle, contradicting Claim~\ref{OneInATriangle}, or an $H_7$-gadget with end $u_2$. But then $v$ is in that $H_7$-gadget and thus in a triangle $T$.

By Claim \ref{OneInATriangle}, $v$ has at least one neighbor of degree at least four. Since $v$ has at least three neighbors of degree three, it follows that the other neighbor $w$ in $T$ has degree three. 
If $w$ is not in the same component of $D_3(G)$ as $u_1$, then it follows that $w$ is in a component of $D_3(G)$ of size at most two. 
Hence $v$ receives at most $+\frac{1}{4}$ from $w$. But then $\ch^*(v)\le -1 + \frac{1}{4} + 2\cdot\frac{3}{8} = 0$, a contradiction.
So $w$ is in the same component of $D_3(G)$ as $u_1$. Thus $w$ is adjacent to $u_3$??

Applying Claim \ref{Gadget} by taking $u=u_3$, $a=w$ and $v=u_2$, we have that $v$ and $u_1$ are in some triangle and $w$ and $u_1$ are the only neighbors of $v$ of degree three, since every component of $D_3(G)$ has diameter at most three.
It is a contradiction.

So we may suppose that $v$ does not receive charge under rule 4(b) and hence $v$ receives charge under rule 3(b). 
Since $v$ receives at most $+\frac{1}{3}$ charge from each neighbor of degree three, it follows that $v$ has four neighbors of degree three. 
Thus $v$ is not in a triangle by Claim~\ref{OneInATriangle}. 
Hence $v$ is in an $H_7$-gadget whose end is a neighbor of $v$ by Claim~\ref{Gadget}. 
Thus two of its neighbors are in triangles. 
Since these neighbors are degree three, it follows from Claim~\ref{OneInATriangle} that they are in components of $D_3(G)$ of size one. 
Hence $v$ receives only $+\frac{1}{6}$ charge from these neighbors and $\ch^*(v) \le -1 + 2\cdot\frac{1}{6} + 2\cdot\frac{1}{3} = 0$ as desired.
\end{proof}

If follows from the above claims that 

$$p(G)+T(G) = \sum_v \ch^*(v)\le 0.$$

Since $G$ is a minimum counterexample, $p(G)\ge -1$. 
This implies that $T(G)\le 1$. 
Hence $G$ does not contain an $H_7$-gadget.
Therefore, every component of $D_3(G)$ has diameter at most two by Claims~\ref{OneInATriangle} and \ref{Gadget}. 

Suppose there exists a component with diameter two. 
It follows from Claim~\ref{Gadget} that $G$ has two triangles $T_1,T_2$, where each of them contains a leaf of the component. So $T(G)= 1$. Since $T(G)\le 1$, $T_1$ and $T_2$ intersect in at least one vertex. 
Yet $T_1$ and $T_2$ intersect in at most one vertex since $G$ does not contain a $K_4-e$ by Claim~\ref{NoK4e}. 
Let $v$ be the unique vertex contained in both $T_1$ and $T_2$. 
By Claim~\ref{OneInATriangle}, $v$ has at least two neighbors of degree at least four. 
Since discharging rule (4) no longer applies, $v$ receives at most $\frac{1}{3}$ charge from its neighbors of degree three. 
But then $\ch^*(v) \le 5-\frac{3}{2}d(v) + \frac{d(v)-2}{3} = \frac{13}{3} - \frac{7}{6}d(v)$. 
As $d(v)\ge 4$, it follows that $\ch^*(v)\le -\frac{1}{3}$. 
Since $\sum_v \ch^*(v)$ is integral, we find that $\sum_v \ch^*(v)\le -1$ and hence $p(G)\le -1 -T(G) = -2$, a contradiction.

So every component of $D_3(G)$ is an edge or vertex. 
Note in this case that discharging rules 3 and 4 do not apply. Let $m$ be the number of edges with both ends of degree at least four. Then $p(G)\le \sum_v ch^*(v) \le \sum_{v\in V(G)-V(D_3(G))} ch^*(v) \le -m/2 +\sum_v( 5-3d(v)/2 +d(v)/4 ) \le -m/2$. So $m\le 2$ as otherwise, $p(G) \leq \sum_v \ch(v) <-1$, a contradiction. Furthermore if $m=2$, then all the equalities hold. In that case, $ch^*(v)=0$ for all $v\in D_3(G)$ and hence every component of $D_3(G)$ is an edge.

First suppose $m=0$. Then we can proper color $G$ by three colors by coloring $G\setminus D_3(G)$ with one color and coloring $D_3(G)$, which is bipartite, with two colors, a contradiction. 
Next suppose $m=1$ and let $f=u_1u_2$ be the unique edge between vertices of degree at least four. 
Then $G[D_3(G) \cup \{u_1\}]$ is bipartite by Claim \ref{OneInATriangle}.
So we can obtain a proper $3$-coloring of $G$ by coloring $G\setminus (D_3(G)\cup\{u_1\})$ with one color and $G[D_3(G)\cup\{u_1\}]$ with two colors, a contradiction. 
Consequently, $m=2$ and there are exactly two edges $f_1=u_1u_2$ and $f_2=u_3u_4$ between vertices of degree at least four. 

Recall that every component of $D_3(G)$ is an edge in this case. We may suppose without loss of generality that $u_1\ne u_3$. 
Now we color $G\setminus (D_3(G)\cup\{u_1,u_3\})$ with color 1, $u_1$ with color 2 and $u_3$ with color $3$. 
Then we can extend this coloring to a coloring of $D_3(G)$ as follows. Let $vw$ be an edge of $D_3(G)$. If at most one of $v$ or $w$ in $N(u_1)\cup N(u_3)$, color that vertex different from its colored neighbors and then extend the coloring to the other vertex, which is possible since it has two available colors (2 and 3). So suppose that both $v$ and $w$ are in $N(u_1)\cup N(u_3)$. Since $v$ and $w$ are not in triangle together, we may suppose without loss of generality that $v\in N(u_1)\setminus N(u_3)$ and $w\in N(u_3)\setminus N(u_1)$. Now color $v$ with color 3 and $w$ with color 2. In this way the coloring can be extended to all the components of $D_3(G)$ and hence $G$ has a $3$-coloring, a contradiction. This proves Theorem~\ref{Main}.

\section{Concluding remarks}
One may wonder if the asymptotic edge-density of $4$-critical graphs may be improved above $5/3$. The second author~\cite{Postle} confirmed this by proving the following theorem:

\begin{theorem}[Postle~\cite{Postle}]\label{Ore4Better}
There exists $\epsilon, t> 0$ such that if $G$ is a $4$-critical graph, then 

$$|E(G)| \ge \frac{(5+\epsilon)|V(G)| - 2 + \epsilon(t-4-tT(G))}{3}$$
\end{theorem}

When $G$ is girth at least five, Theorem~\ref{Ore4Better} provides the following corollary.

\begin{corollary}[Postle~\cite{Postle}]\label{Ore4GirthFive}
There exists $t,\epsilon> 0$ such that if $G$ is a $4$-critical graph of girth at least five, then 

$$|E(G)| \ge \frac{(5+\epsilon)n - 2 + (t-4)\epsilon}{3}$$
\end{corollary}

Corollary~\ref{Ore4GirthFive} implies that for large $4$-critical graphs of girth at least five, the number of edges differs greatly from $\frac{5}{3}|V(G)|$. However, Theorem~\ref{Ore4GirthFive} does not imply our main result. Nevertheless we believe the two theorems could be merged to provide one unified theorem as well as better value for $\epsilon$.

On the other hand, the condition in Corollary~\ref{Ore4GirthFive} for girth five cannot be replaced by girth four. A construction of Thomas and Walls~\cite{} using Ore-compositions shows that the asymptotic density is 5/3 for triangle-free 4-critical graphs. It would be of interest to answer the following question then:

\begin{question}
What is the minimum $c$ such that there exists a triangle-free $4$-critical graph on $n$ vertices with $\frac{5n+c}{3}$ edges?
\end{question}

Kostochka and Yancey's bound shows that $c\ge -2$. Our main result however cannot improve this since $T(G)$ counts $4$-cycles. The proof would have to be modified to avoid this pitfall. The best upper bound we know is $c\le 5$ as evidenced by the Mycelski graph on 11 vertices. We conjecture that this is correct. Similarly it would be of interest to find the minimum for graphs of girth five:

\begin{question}
What is the minimum $c$ such that there exists a $4$-critical graph of girth five on $n$ vertices with $\frac{5n+c}{3}$ edges?
\end{question}

Our main result implies that $c\ge 2$, but we think the number should be higher. On the other hand, since there exists a 4-critical 4-regular graph on 21 vertices (the so-called Grunbaum graph), $c\le 21$. We believe that with further work, our methods should solve this question. 
Namely, by ``digging deeper" in the list of graph potentials and categorizing the graphs of potentials $-2, -3, \ldots$, one should be able to find the best $c$. Of course the list of such graphs would become more numerous but is still finite. On the other hand, the discharging part of our proof would have to be strengthened and new analysis developed to show that the sum of the charges is at most the negative of that best possible $c$.


\begin{thebibliography}{99}
\bibitem{Postle} L. Postle, The edge-density of $4$-critical graphs of girth 5 is $\frac{5}{3}+\epsilon$, manuscript.
\end{thebibliography}
\end{document}